\documentclass[11pt,a4paper]{amsart}
\usepackage[utf8]{inputenc}
\usepackage{amsmath}
\usepackage{amsfonts}
\usepackage{amssymb}
\usepackage{amsmath}
\usepackage{amsthm}
\usepackage{amssymb}
\usepackage{cite}
\usepackage[shortlabels]{enumitem}	
\usepackage{xcolor}
\usepackage{tikz-cd}
\usepackage{float}
\usepackage{url}
\usepackage{hyperref}
\usepackage{graphicx}
\newtheorem{thm}{Theorem}[section]
\newtheorem{theorem}{Theorem}
\newtheorem{cor}[thm]{Corollary}

\newtheorem{pr}[thm]{Proposition}
\theoremstyle{definition}
\newtheorem{df}[thm]{Definition}
\theoremstyle{remark}
\newtheorem{rem}[thm]{Remark}

\usepackage[a4paper, left = 3cm, right = 3cm, top = 2.5cm, bottom = 2.5cm, headsep = 1.2cm]{geometry}
\begin{document}
	
	\title[Computable paradoxical decompositions]{Computable paradoxical decompositions}

\author{Karol Duda$^{\dag}$}
\address{Faculty of Mathematics and Computer Science,
	University of Wroc\l aw\\
	pl.\ Grun\-wal\-dzki 2,
	50--384 Wroc\-{\l}aw, Poland}
	\email{karol.duda@math.uni.wroc.pl}
\thanks{$\dag$ Partially supported by (Polish) Narodowe Centrum Nauki, UMO-2018/30/M/ST1/00668.}
\author{Aleksander Ivanov (Iwanow)}
\address{ Department of Applied Mathematics, Silesian University of Technology, ul. Kaszubska 23\\
Gliwice, 44-100, Poland} 
\email{Aleksander.Iwanow@polsl.pl} 

\begin{abstract}
We prove a computable version of Hall's Harem Theorem and apply it to computable versions of Tarski's alternative theorem.    
\end{abstract}

\maketitle

\section{Introduction} 

The Hall harem theorem describes a condition which is equivalent to existence of a perfect $(1,k)$-matching of a bipartite graph, see Theorem H.4.2 in \cite{csc}. 
When $k=1$ this is exactly Hall's marriage theorem, see Section III.2 in \cite{BB}. 
These theorems are useful in amenability.  
For example some versions of Tarski's alternative theorem can be obtained in this way, see Chapter 4 in \cite{csc}  and Section III.1 in \cite{CGH} .   
In \cite{hak} Kierstead found a computable version of Hall's marriage theorem. 
In this paper we generalize his theorem for arbitrary $k$ and give an application of this generalization to effective amenablity. 
 
To introduce the reader to the subject we recall the following definition. 

\begin{df} \label{PD}
Let $X$ be a set and let $G$ be a group which acts on $X$ by permutations. 
The $G$-space $(G,X)$ has a paradoxical decomposition, if there exists a finite set $K\subset G$ and two families  
$(A_k)_{k\in K}$ and $(B_k)_{k\in K}$ of subsets of $X$ such that 
\[ 
X = \Big(\bigsqcup\limits_{k\in K}k(A_k) \Big)\bigsqcup\Big(\bigsqcup\limits_{k\in K}k(B_k)\Big)=\Big(\bigsqcup\limits_{k\in K}A_k\Big)=\Big(\bigsqcup\limits_{k\in K}B_k\Big).
\] 
We call $(K,(A_k)_{k\in K}, (B_k)_{k\in K})$ a paradoxical decomposition of $X$.
\end{df}
Here we use a version of the definition given in \cite{csc}, where some members $A_k$ or $B_k$ can be empty. 
It is equivalent to the traditional one. 
A well-known theorem of A. Tarski \cite{tar} states that the existence of such a paradoxical decomposition is  opposite to amenability of the $G$-space $(G,X)$. 
In particular a group is amenable if and only if it does not admit a paradoxical decomposition. 

It is worth noting that there is a variety of versions of this theorem in different contexts, see for example \cite{MU}, \cite{pat}, \cite{schn} and  \cite{ST}. 
In this paper we will study ones which are natural from the point of view of computability theory, \cite{sri}.  
In the situation when $X= \mathbb{N}$ and $G$ acts by computable permutations one can additionally demand that the families $(A_k )_{k\in K}$ and $(B_k )_{k\in K}$ consist of computable sets. 
We call such a paradoxical decomposition {\em computable}. 

One of versions of Tarski's theorem concerns a very general situation of {\em pseudogroups of transformations}. 
The following definition is taken from \cite{CGH} and \cite{HS1}. 

\begin{df} 
A pseudogroup $\mathcal{G}$ of transformations of a set $X$ is a set of bijections 
$\rho : S \rightarrow  T$ between subsets $S$ and $T \subseteq X$ which satisfies the
following conditions:  \\ 
(i) the identity $id_X$ is in $\mathcal{G}$, \\ 
(ii) if $\rho : S \rightarrow T$ is in $\mathcal{G}$, so is the inverse 
$\rho^{-1} : T \rightarrow S$, 
\\ 
(iii) if 
$\rho_1 : S \rightarrow T$ and $\rho_2 : T \rightarrow U$ 
are in $\mathcal{G}$, so is their composition $\rho_2 \circ \rho_1 : S \rightarrow U$, \\ 
(iv) if $\rho : S \rightarrow T$ is in $\mathcal{G}$ and if $S_0$ is a subset of $S$, the restriction $\rho | S_0$ is in $\mathcal{G}$,  \\ 
(v) if $\rho : S \rightarrow T$ is a bijection between two subsets $S, T$ of $X$ and if there exists a finite partition $S = \bigcup_{j\le n} S_j$ with $\rho | S_j \in \mathcal{G}$ for $j \in {1, . . . , n}$, then  $\rho$ is in $\mathcal{G}$.
\end{df} 

For $\gamma:S\rightarrow T$ in $\mathcal{G}$, we write $\alpha(\gamma)$  for the domain $S$ of $\gamma$ and $\omega(\gamma)$ for its range $T$. 

\begin{df}
When $X$ is countable, after identifying $X$ with $\mathbb{N}$, we say that a transformation 
$\rho : S\rightarrow T$ from $\mathcal{G}$ is computable if $S$ and $T$ are computable subsets of $\mathbb{N}$ and $\rho$ is a computable function. 
\end{df}

Note that for any tuples $( a_1 , \ldots , a_k )$ and $(b_1 , \ldots , b_k )$ with pairwise distinct  coordinates where each $b_i$ is in the same $\mathcal{G}$-orbit with the corresponding $a_i$, the map 
$(a_1 , \ldots , a_k )\rightarrow (b_1 , \ldots , b_k )$ is a computable transformation from $\mathcal{G}$. 

A typical illustration of these notions appears in the case of discrete  metric spaces. 
We remind the reader that given a metric space $(X,d)$ and a subset $F\subseteq X$ the set 
$N_m (F)=\{x\in X \, | \, d(x, F )\leq m\}$ is called the $m$-{\em ball} of $F$. 
A metric space $X$ is called {\em discrete} if the 1-ball of  every finite subset is finite. 

\begin{df} 
For a metric space $X$, the  pseudogroup $W(X)$ of bounded perturbations of the identity consists of bijections $\rho : S \rightarrow T$ such that 
$\mathsf{sup}_{x\in S}( d(\rho(x), x))$ is bounded by some natural number (depending on $\rho$). 
It is called the pseudogroup of wobbling bijections. 
\end{df} 

When $X$ is infinite and discrete the values $\mathsf{sup}_{x\in S}( d(\rho(x), x))$ for $\rho \in W(X)$ are not uniformly bounded by a natural number.  

\begin{df}
When $X$ is conutable, then after identifying $X$ with $\mathbb{N}$, the  effective wobbling pseudogroup   $W_{eff}(X)$ of $X$ is a subset of $W(X)$ consisting of computable transformations of $X$. 
\end{df}

We now formulate one of the definitions of amenability. 
Let $\mathcal{G}$ be a pseudogroup of transformations of $X$. 
For $R\subset \mathcal{G}$ and  $A\subset X$ we define the 
$R$-boundary of $A$ as
\[ 
\partial_R A=\{x\in X\setminus A \text{ : } \exists \rho \in R\cup R^{-1} 
( x\in \alpha(\rho) \text{ and } \rho(x)\in A )\} . 
\]  
\begin{df} \label{pseF} 
The pseudogroup $\mathcal{G}$ satisfies the F\o lner condition if for any finite subset $R$ of $\mathcal{G}$ and any natural number $n$ there exists a finite non-empty subset $F=F(R,n)$ of $X$ such that 
$|\partial_R F|<\frac{1}{n} |F|$. 
\end{df} 
The following theorem is a version of Tarski's theorem mentioned above, 
see Theorems 7 and 25 in \cite{CGH}. 
\begin{itemize} 
\item  The pseudogroup $\mathcal{G}$ satisfies the F\o lner condition if and only if there is no tuple $(X_1,X_2,\gamma_1,\gamma_2)$ consisting of a non-trivial partition $X=X_1\sqcup X_2$ and $\gamma_i\in \mathcal{G}$ with $\alpha(\gamma_i)=X_i$ and $\omega(\gamma_i)=X$ for $i=1,2$. 
\end{itemize} 

\begin{rem} \label{actF}  
Definition \ref{pseF} can be applied to an action of a group $G$ on 
a set $X$ by permutations. 
In this case we will say that the $G$-space $(G,X)$ 
satisfies F\o lner's condition.   
\end{rem} 
 
The motivation for computable versions of this theorem comes from recent 
investigations in effective amenability theory, \cite{MC2}, \cite{MC3} 
and \cite{mor}, where some effective versions of F\o lner's condition were suggested.  
Our main result connects this approach with paradoxical decompositions. 
 
In Section 2 we generalize the work of Kierstead  \cite{hak} 
concerning an effective version of Hall's Theorem. 
These results will be applied in Section 3  
to some computable versions of Tarski's alternative theorem.
In Section 4 we study some complexity issues which are naturally connected with the main results of the paper. 

We do not demand any special education of the reader in computability theory. 
Facts which we use are well-known and easily available in \cite{sri}. 
Following trends in logic we say computable instead of recursive.   

\section{A computable version of Hall's Harem Theorem} 

A graph $\Gamma=(V,E)$ is called a {\em bipartite graph} if the set of vertices $V$ is partitioned into sets $A$ and $B$ in such way, that the set of edges $E$ is a subset of $A\times B$. We denote such a bipartite graph by $\Gamma=(A,B,E)$. The set $A$ (resp. $B$) is called the set of {\em left} (resp. {\em right}) {\em vertices}.
From now on we concentrate on bipartite graphs. 
Although our definitions concern this case they usually have obvious extensions to all ordinary graphs. 

Let $\Gamma=(A,B,E)$. 
When $(a,b)$ is an edge from $E$, it is called {\em adjacent} to vertices $a$ and $b$. 
In this case we say that $a$ and $b$ are adjacent too. 
When  two edges $(a,b),(a',b')\in E$ have a common adjacent vertex we say that $(a,b),(a',b')$ are also {\em adjacent}.  
A sequence $( a_1 , a_2 , \ldots , a_n )$ of vertices is called a {\em path} if each pair $( a_i$, $a_{i+1})$ is adjacent for $1 \le i\le n$. 

Given a vertex $x\in A\cup B$ the {\em neighbourhood} of  $x$ is the set 
\[ 
N _{\Gamma}(x)=\{y\in A\cup B: (x,y)\in E\}.
\]  
For subsets $X\subseteq A$ and $Y\subseteq B$, we define the neighbourhood 
$N _{\Gamma}(X)$ of $X$ and the neighbourhood $N _{\Gamma}(Y)$ of $Y$ by
\[
N _{\Gamma}(X)=\bigcup\limits_{x\in X} N _{\Gamma}(x) \subseteq B  \, \text{ and } \, N _{\Gamma}(Y)=\bigcup\limits_{y\in Y} N _{\Gamma}(y)\subseteq A.
\] 
The subscript $\Gamma$ is dropped if it is clear from the context.

In this section we always assume that $\Gamma$ is {\em locally finite},  i.e. the set $N(x)$ is finite for all $x\in A\cup B$. 

A subset $X$ of $A$ (resp. of $B$) is called {\em connected} if for all $x, x' \in X$ there exist a path $( p_1,\ldots, p_k )$ in $\Gamma$ with $x=p_1$ and $x'=p_k$ such that $p_i\in X\cup N_{\Gamma}(X)$ for all $i\le k$.

For a given vertex $v\in A\cup B$ the {\em star} of $v$ is a subgraph $S=(V',E')$ of $\Gamma$, with $V'=\{v\}\cup N_{\Gamma}(v)$ and $E'=(V'\times V' )\cap E$.

\begin{df} 
A \textsl{matching} ($(1,1)$-matching) for $\Gamma$ 
is a subset $M\subset E$ of pairwise nonadjacent edges. 
A matching $M$ is called \textsl{left-perfect} 
(resp. \textsl{right-perfect}) 
if for all $a \in A$ (resp. $b\in B$)  there exists (exactly one) 
$b\in B$ (resp. $a\in A$) with $(a,b)\in M$.
The matching $M$ is called \textsl{perfect} if it is both right and left-perfect.
\end{df} 

We now introduce perfect $(1,k)$-matchings for $\Gamma$ without defining $(1,k)$-matchings. 
We will use only perfect ones.

\begin{df}
A perfect $(1,k)$-matching for $\Gamma$ is a subset $M\subset E$ satisfying the following conditions: 

\begin{enumerate}[(1)]
\item for all $a \in A$ there exist exactly $k$ vertices $b_1,\ldots b_k \in B$ such that \newline
 $(a,b_1),\ldots,(a,b_k)\in M$;

\item for all $b \in B$ there is a unique vertex $a\in A$ such that $(a,b)\in M$. 
\end{enumerate} 
\end{df}
Given a $(1,k)$-matching $M$ and a vertex $a\in A$ the $M$-{\em star} of $a$ is the graph consisting of all vertices and edges adjacent to $a$ in $M$. 

The following Theorem is known as {\em the Hall harem theorem}, 
and the first of equivalent conditions below is known as {\em Hall's $k$-harem condition}, see Theorem H.4.2 in  \cite{csc}.

\begin{thm} \label{H42} 
Let $\Gamma=(A,B,E)$ be a locally finite graph and let $k\in \mathbb{N},\; k\geq 1$. The following conditions are equivalent: 
\begin{enumerate}[(i)]
\item For all finite subsets $X\subset A$, $Y\subset B$ the following inequalities hold: 
\newline
$|N(X)|\geq k|X|$, $|N(Y)|\geq \frac{1}{k}|Y|$.

\item $\Gamma$ has a perfect $(1,k)$-matching.
\end{enumerate} 
\end{thm}

In order to define computable versions of these conditions we follow Kierstead's paper \cite{hak}.  
Definitions \ref{kie1} - \ref{kie3} are due to Kierstead. 
Definitions \ref{cpkm} and \ref{kie4} are natural generalizations of the corresponding ones from \cite{hak}.   
 
\begin{df} \label{kie1}
A graph $\Gamma =(V,E)$ is computable if there exists a bijective function $\nu: \mathbb{N}\rightarrow V$ such that the set 
\[
R:=\{(i,j): (\nu(i),\nu(j))\in E\} 
\] 
is computable. 
\end{df} 

\begin{df} \label{kie2}
A bipartite graph $\Gamma=(A,B,E)$ is computably bipartite  
if $\Gamma$ is computable as a graph with respect to some $\nu$ and the set $\nu^{-1}(A) = \{ n\in \mathbb{N}: \nu (n) \in A \} \subset \mathbb{N}$ is computable.
\end{df}

To simplify the matter below we will always identify $A$ and $B$ with $\mathbb{N}$. 
Thus $A$ (resp. $B$) will be called the left (resp. right) copy of $\mathbb{N}$ and the function $\nu$ will be the identity map.

\begin{df} \label{kie3} 
A locally finite (bipartite) graph $\Gamma$ is called highly computable if it is  computable and the function 
$n \rightarrow |N_{\Gamma}(n)|$ for $n\in\mathbb{N}$ is computable. 
\end{df}

\begin{df}\label{cpkm}
Let $\Gamma=(A,B,E)$ be a computably bipartite graph. 
A perfect $(1,k)$-matching $M$ for $\Gamma$ is called 
computable if the set $\{ (i,j): (\nu(i), \nu(j))\in M\}\subset \mathbb{N} \times \mathbb{N}$ is computable. 
\end{df} 

Note that computable perfectness exactly means that there is  an algorithm which 
\begin{itemize}
\item for each $i \in A$, finds the tuple $(i_1,i_2, \ldots, i_k)$ such that $(i,i_j )\in M$, for all $j=1,2,\ldots, k$; 
\item when $i\in B$ it finds $i'\in A$ such that $(i',i)\in M$.
\end{itemize}

The remainder of this section will be devoted to a proof that the following condition implies the existence of a computable perfect $(1,k)$-matching.

\begin{df} \label{kie4} 
A highly computable bipartite graph $\Gamma=(A,B,E)$ satisfies 
the computable expanding Hall's harem condition with respect to $k$  
(denoted $c.e.H.h.c.(k)$), if and only if there is a  computable  function $h: \mathbb{N} \rightarrow \mathbb{N}$ with domain $\mathbb{N}$ such that:
\begin{itemize}
\item $h(0)=0$
\item for all finite sets $X\subset A$, the inequality $h(n)\leq |X|$ implies $n\leq |N(X)|-k|X|$
\item for all finite sets  $Y\subset B$, the inequality $h(n)\leq |Y|$ implies $n\leq |N(Y)|-\frac{1}{k}|Y|$.
\end{itemize}
\end{df} 

Clearly, if the graph $\Gamma$ satisfies the $c.e.H.h.c.(k)$, then it satisfies Hall's $k$-harem condition. 
We emphasize that the requirements that $h$ is total and computable, essentially strengthen the latter ones. 
Moreover, Theorems 2 and 5 of \cite{hak} state that the natural effective version of Hall's marriage theorem (i.e. when $k=1$) does not hold without the assumptions that $h$ exists and is computable. 
It is worth noting that Theorem 2 of \cite{hak} is a citation of a result of Manaster and Rosenstain from \cite{MR}. 

\begin{thm} \label{K} 
If $\Gamma=(A,B,E)$ is a highly computable bipartite graph satisfying the $c.e.H.h.c.(k)$, then $\Gamma$ has a computable perfect $(1,k)$-matching.
     
\end{thm} 

\begin{proof}
We extend the proof of Theorem 3 of Kierstead's paper \cite{hak}. 
Let $h$ witness the $c.e.H.h.c.(k)$ for $\Gamma$. 
Let us fix computable enumerations of $A$ and $B$. 
We build a perfect $(1,k)$-matching $M$ by induction.  
The idea of the construction is as follows. 
At step 0 put $M=\emptyset$.
At step $s$ we update the already constructed $M$ in the following way. For the first vertex $x_s$ from the remaining part of $A$ or $B$ we construct some finite subgraph $\Gamma_s$ and a matching $M_s$ in $\Gamma_s$. 
The matching $M$ is updated by adding the elements of $M_s$ adjacent to $x_s$. 
The subgraphs $\Gamma_s$ and $M_s$ are constructed so that after removal of the $M_s$-star of $x_s$ from $\Gamma$, the remaining part  still is a highly computable bipartite graph satisfying the $c.e.H.h.c.(k)$.
  
At the first step of the algorithm we choose $a_0$, the first element of the set $A$. 
We construct the induced subgraph $\Gamma_0=(A_0,B_0,E_0)$ so that $A_0\cup B_0$ is the set of vertices of distance of at most $\max\{2h(k)+1,3\}$ from $a_0$. 
Since the graph $\Gamma$ is locally finite (resp. highly computable) the graph $\Gamma_0$ is finite and can be found effectively. 
It is clear that for all vertices $v$ from $A_0$, $N_{\Gamma_{0}}(v)=N_{\Gamma}(v)$. 
Therefore, for every subset $X\subset A_0$ the inequality $h(n)\leq |X|$ implies $n\leq |N_{\Gamma_{0}}(X)|-k|X|$.

Let $B_{S_0}$ denote the set of vertices $v\in B_0$ at distance $\max\{2h(k)+1,3\}$ from $a_0$. 
It is clear that $N_{\Gamma_{0}}(B_0\setminus B_{S_0})=N_{\Gamma}(B_0\setminus B_{S_0})=A_0$. 
On the other hand since it may happen that $N_{\Gamma}(B_{S_0})$ is not contained in $A_0$, it is possible that there exists a subset $Y\subset B_{S_0}$, such that $|N_{\Gamma_{0}}(Y)|\leq\frac{1}{k}|Y|$. 

Since $\Gamma$ contains a perfect $(1,k)$-matching, there exists a $(1,k)$-matching in $\Gamma_0$, that satisfies the conditions of perfect $(1,k)$-matchings for all $a\in A_0$ and $b\in B_0 \setminus B_{S_0}$. 
We denote it by $M_0$. 
Since $\Gamma_0$ is finite, the matching $M_0$ can be obtained effectively. 
Let $\{(a_0 ,b_{0,1}),\ldots, (a_0 ,b_{0,k})\}$ be the set of all edges from $a_0$ which belong to $M_0$. 
At step 1 we define $M$ to be the set of all these pairs.

Let $\Gamma'$ be the subgraph (yet bipartite) obtained from $\Gamma$ through removal of the $M_0$-star of $a_0$.
Since the sets $A\cup B$, $A$ and $E$ are computable, and the matching $M_0$ is found effectively, the sets $A'$, $B'$ and $E'$ are also computable.
Therefore $\Gamma'$ is a computably bipartite graph.
Since $\Gamma$ is highly computable, the graph $\Gamma'$ is highly computable too.
To finish this step it suffices to show that $\Gamma'$ satisfies $c.e.H.h.c.(k)$.

Define $h' :\mathbb{N} \rightarrow \mathbb{N}$ by setting 
\[ 
h'(n) = \left\{
\begin{array}{rr}
0, \quad \text{if} \quad n=0, \ \\
h(n+k), \quad \text{if} \quad n > 0.
\end{array}\right.
\] 
We claim that $h'$ works for $\Gamma'$. 
We start with the case when $X\subset A'$ and $n>0$.
Since $|N_{\Gamma'}(X)|\geq |N_{\Gamma}(X)|-k$, then for $n\geq 1$ the inequality $|X|>h'(n)$ implies 
$|N_{\Gamma'}(X)|-k|X|\geq |N_{\Gamma}(X)|-k|X|-k\geq n$.

Let us consider the case when $n=0$ and $X$ is still a subset of $A'$. 
If $X$ is not connected, then its neighbourhood would be the union of nieghbourhoods of its connected subsets. 
Therefore without loss of generality, we may assume that $X$ is connected. 
If $X\subset A_0$, then $|N_{\Gamma'}(X)|- k|X| \geq 0$, since $M_0$ was a $(1,k)$-matching for $\Gamma_0$ that was perfect for subsets of $A_0$. 

Now, let $a'\in X\setminus A_0$.  
If $b_{0,1},\ldots, b_{0,k} \notin N_{\Gamma}(X)$, then 
$N_{\Gamma'}(X) = N_{\Gamma}(X)$, so $|N_{\Gamma'}(X)|-k|X|\geq 0$. 
Assume that for some $i\leq k$ and some $a\in X,$ there exists $(a,b_{0,i})\in E$. 
Since the distance between $a$ and $a'$ is at least $2h(k)$ we have $|X|\geq h(k)+1$. 
Thus $|N_{\Gamma}(X)|-k|X|\geq k$ and it follows that $|N_{\Gamma'}(X)|-k|X|\geq 0$. 
We conclude that the Hall condition for finite subsets of $A'$ is verified. 

Now we need to show that $\Gamma'$ satisfies 
$c.e.H.h.c.(k)$ for finite sets $Y\subset B'$. 
We have to show that the inequality $h'(n)\leq |Y|$ implies $n\leq |N_{\Gamma'}(Y)|-\frac{1}{k}|Y|$.
Note $Y\subset B'=B \setminus \{b_{0,1},\ldots, b_{0,k}\}$ 
and $|N_{\Gamma'}(Y)|\geq |N_{\Gamma}(Y)|-1$.

In the case $n\!>\!0$ the inequality $|Y|>h'(n)$ 
implies \!
$|N_{\Gamma'}(Y)|-\frac{1}{k}|Y|\geq |N_{\Gamma}(Y)|-\frac{1}{k}|Y|-1\geq n+k-1\geq n$.

Let us consider the case $n=0$. 
As before, we may assume that $Y$ is connected.
If $Y\subset B_0\setminus B_{S_0}$, then 
$|N_{\Gamma'}(Y)|- \frac{1}{k}|Y| \geq 0$,  since $M_0$ satisfied the conditions of a perfect $(1,k)$-matching for elements of  $B_0\setminus B_{S_0}$. 
If $a_0\notin N_{\Gamma}(Y)$, then 
$N_{\Gamma'}(Y)=N_{\Gamma}(Y)$ and again $|N_{\Gamma'}(Y)|- \frac{1}{k}|Y| \geq 0$. 

Assume that there exists $b'\in Y\setminus (B_0\setminus B_{S_0})$ and there exists $b\in Y $ with the edge $(a_0,b)\in E$.
Since the distance between $b$ and $b'$ is at least $2h(k)$ we have $|Y|\geq h(k)+1$. 
It follows that 
$|N_{\Gamma}(Y)|- \frac{1}{k}|Y| \geq k$ and $|N_{\Gamma'}(Y)|-\frac{1}{k}|Y|\geq k-1\geq 0$.
As a result we have that the graph $\Gamma'$ satisfies $c.e.H.h.c.(k)$. 

To force $M$ to be a perfect $(1,k)$-matching, we use back and forth. 
Therefore we start the next step of our algorithm by choosing the first element  of $B'$, say $b_{1,1}$. 
We construct the induced subgraph $\Gamma_1=(A_1,B_1,E_1)$ so that $A_1\cup B_1$ is a set of vertices of $\Gamma'$ at distance at most $\max\{2h'(k)+2,4\}$ from $b_{1,1}$.
Let $B_{S_1}$ denote the set of vertices at distance $\max\{2h'(k)+2,4\}$ from $b_{1,1}$.
Since $\Gamma'$ contains a perfect $(1,k)$-matching, there exist a $(1,k)$-matching in $\Gamma_1$ that satisfies the conditions of a perfect $(1,k)$-matching for all $a\in A_1$ and $b\in B_1\setminus B_{S_1}$. 
We denote it by $M_1$. 
We choose $a_1$ with $(a_1,b_{1,1})\in M_1$. 
Let $\{ (a_1 ,b_{1,2}),\ldots, (a_1 ,b_{1,k})\}$ be all remaining edges of the $M_1$-star of  $a_1$. 
We update $M$ by adding all edges of this star.

Let $\Gamma''$ be a subgraph obtained from $\Gamma'$ through removal of the $M_1$-star of $a_1$. 
Then $\Gamma''$ is also a highly computable computably bipartite graph. 
We need to show that $\Gamma''$ satisfies $c.e.H.h.c.(k)$.

Define $h'': \mathbb{N} \rightarrow \mathbb{N}$ by setting 
\[
h''(n) = \left\{
\begin{array}{rr}
0, \quad \text{if} \quad n=0, \ \\
h'(n+k), \quad \text{if} \quad n > 0.
\end{array}\right.
\] 
To prove that $h''(n)$ works for $\Gamma''$ we use the same method as in the case $h'(n)$ and $\Gamma'$.

We continue iteration by taking the elements of $A$ at even steps and 
the elements of $B$ at odd steps. 
At every step $n$, the graph $\Gamma^{(n)}$ satisfies the conditions for the existence of perfect $(1,k)$-matchings and we update $M$ by adding $k$ edges adjacent to $a_n$. 
Every vertex $v$ will be added to $M$ at some step of the algorithm. 
It follows that $M$ is a perfect $(1,k)$-matching of the graph $\Gamma$. 
Effectiveness of our back and forth construction guarantees that  $M$ is computable. 
\end{proof}

\section{Effective paradoxical decomposition}

The following definition gives an effective version of a paradoxical decomposition. 
Assume that a pseudogroup $\mathcal{G}$ acts on a countable set $X$.  
We will identify $X$ with $\mathbb{N}$. 

\begin{df}
Let $\mathcal{G}$ be  a pseudogroup of transformations of a set $X = \mathbb{N}$. 
An effective paradoxical $\mathcal{G}$-decomposition of $(\mathcal{G},X)$ is a tuple $(X_1,X_2,\gamma_1,\gamma_2)$ consisting of a non-trivial partition $X=X_1\sqcup X_2$ into computable sets and computable $\gamma_i\in \mathcal{G}$ with $\alpha(\gamma_i)=X_i$ and $\omega(\gamma_i)=X$ for $i=1,2$. 
\end{df} 
We now formulate the main theorem of this section. 

\begin{thm}\label{pdx} 
Let $(\mathcal{G},X)$ be a pseudogroup of computable transformations  defined on $\mathbb{N}$ which does not satisfy F\o lner's condition. 
Then $X$ has an effective paradoxical $\mathcal{G}$-decomposition. 
\end{thm} 

\begin{proof} 
This proof is an effective version of Theorem 4.9.2 of \cite{csc}.  
Let $R$ be a non-empty finite subset of $\mathcal{G}$ and let $n$ be a natural number such that for any non-empty finite subset $F$ of $X$ one has $|\partial_R F|\geq\frac{1}{n} |F|$. 
Define a function $d_R$ on $X$ by setting, for all $x,y \in X$, 
\[  
d_R (x, y) = \mathsf{min} \{ n \in \mathbb{N}  
\mbox{ : } \exists \rho_1 ,\ldots ,\rho_n \in R \cup R^{-1} 
\mbox{ ( } 
\rho_n \circ 
\ldots \circ \rho_1 (x) 
\mbox{ is defined } 
\] 
\[ \mbox{ and is equal to } y ) \} , 
\] 
where in the case when there exists no $n$ as in the formula above we put $d_R (x, y) = \infty$. 
The function $d_R$ satisfies the triangle inequality for any triple from $X$. 
Hence we use it as a metric. 
Since $R$ is a finite set of computable transformations, the set $\{ (x,y): d_R (x,y) \le k\}$ is computable uniformly on $k$. 
Therefore there is a computable enumeration of the set  
\[ 
\{ (x,y,l)\in X\times X\times \mathbb{N}: d_R (x,y ) \le l \} .   
\] 
Let $k$ be an integer such that $(1+\frac{1}{n})^k\geq 3$. 
By the choice of $R$, for any finite subset $F$ of the space $(X, d_R )$  we have $|N_1(F)|\geq (1+\frac{1}{n})|F|$.   
Thus in this space the size of the $k$-neighborhood $N_k (F)$ is at least $3|F|$ .

To find the corresponding effective paradoxical decomposition 
consider the bipartite graph $\Gamma(X)=(\mathbb{N}, \mathbb{N}, E)$, 
where the set $E\subset \mathbb{N} \times \mathbb{N}$ consists of all pairs $(x,y)$ 
with $d_R (x,y)\leq k$, with $x,y$ viewed as elements of $X$.
By discreteness of $(X,d_R )$ and computability properties of $d_R$, the graph $\Gamma (X)$ is highly computable.

If $F$ be a finite subset of $\mathbb{N}$  
then $|N_{\Gamma}(F)|=|N_k (F)|\geq 3|F|$. 
It follows that: 
\[ 
|N_{\Gamma}(F)|-2|F|\geq 3|F|-2|F|=|F|.
\]  
Therefore for any $n\in \mathbb{N}$ and a finite subset $F$ of the left side of $\Gamma (X)$ the inequality $n\leq |F|$ implies that $n\leq |N_{\Gamma}(F)|-2|F|$.
On the other hand viewing $F$ as a subset of the right side we have 
\[ 
|N_{\Gamma}(F)|- \frac{1}{2} |F| \geq 3|F|- \frac{1}{2}|F| \geq |F|.
\]  
Since the function $h=\mathsf{id}$ is computable, the graph 
$\Gamma (X)$ satisfies $c.e.H.h.c.(2)$ with respect to $h$.
By virtue of the Effective Hall Harem Theorem (Theorem \ref{K}), we deduce the existence of a computable perfect $(1,2)$-matching $M$ in $\Gamma (X)$. 
In other words, there is a computable surjective map 
$\phi: \mathbb{N} \rightarrow \mathbb{N}$ which is 
a 2-to-1 map with the condition that 
$d_R (x,\phi (x)) \le k$ for all $x \in X$.   

We now define functions $\psi_1, \psi_2$ as follows: 
\[ 
\left\{
\begin{array}{r}
\psi_1(n)=\min(n_1,n_2) \\
\psi_2(n)=\max(n_1,n_2)
\end{array}\right. , \text{ where } \phi(n_1)=n=\phi(n_2), n_1\neq n_2.
\] 
Since the function $\phi$ realizes a computable perfect $(1,2)$-matching, 
both $\psi_1$ and $\psi_2$ are computable.

Let $X_i$ be the range of $\psi_i$, $i\in \{ 1,2\}$. 
Clearly, both of them are computable sets and  $X_1\sqcup X_2 = X$. 
We define $\gamma_i:X_i\rightarrow X$ by $\gamma_i(n)=\phi(n)$. 
Since $d_R (x,\gamma_i (x)) \le k$ for all $x \in X$, we have $\gamma_i\in \mathcal{G}$. 
Therefore $(X_1,X_2,\gamma_1,\gamma_2)$ is an effective paradoxical decomposition of $X$.    
\end{proof}

\bigskip 

\begin{cor} 
Let $(X,d)$ be a countable discrete metric space. 
Assume that $W_{eff}(X)$ does not satisfy F\o lner's condition. 
Then $(X,d)$ has an effective paradoxical $W_{eff}(X)$-decomposition. 
\end{cor} 

\bigskip 

In the case of an action of a group $G$ on $X$ we will consider a more precise condition. 

\begin{df} \label{EPD}
Let $X$ be a set identified with $\mathbb{N}$ and let $G$ be a group which acts on $X$ by computable permutations. 
The space $(G,X)$ has a computable paradoxical decomposition, if there exists a finite set $K\subset G$ and two families of computable sets $(A_k)_{k\in K}$, $(B_k)_{k\in K}$ such that:
\[ 
X = \Big(\bigsqcup\limits_{k\in K}k(A_k) \Big)\bigsqcup\Big(\bigsqcup\limits_{k\in K}k(B_k)\Big)=\Big(\bigsqcup\limits_{k\in K}A_k\Big)=\Big(\bigsqcup\limits_{k\in K}B_k\Big).
\] 
We call $(K,(A_k)_{k\in K}, (B_k)_{k\in K})$ a computable paradoxical decomposition of $X$.
\end{df}

Observe that this definition makes sense without the assumption that any element of $G$ realizes a computable permutation of $X$. 
In fact one may demand this only for elements of $K$.  
Since Theorem \ref{EPDthm} does not transcend the assumptions of Definition \ref{EPD} we do not consider the extended version. 
This theorem is a natural development of Theorem \ref{pdx}. 

\bigskip 

\begin{theorem} \label{EPDthm} 
Let $G$ be a group of computable permutations on a countable set $X$ which does not satisfy F\o lner's condition. 
Then there is a finite subset $K \subset G$ which defines 
a computable paradoxical decomposition as in Definition \ref{EPD}.
\end{theorem}

\begin{proof} 
In the beginning of the proof we repeat the argument of  
Theorem \ref{pdx}. 

We denote by $\circ$ the action of $G$ on $X$. 
Find a finite subset $K_0 \subset G$ and a natural number $n$ such that  for any finite subset $F\subset X$, 
there exists $g \in K_0$ such that $\frac{|F\setminus g\circ F|}{|F|}\geq \frac{1}{n}$. 
We may assume that $K_0$ is symmetric. 
Let $R = K_{0}\cup\{1\}$ and let a function 
$d_R$ be defined exactly as in the proof of Theorem \ref{pdx}: 
\[  
d_R (x, y) = \mathsf{min} \{ n \in \mathbb{N}  
\mbox{ : } \exists \rho_1 ,\ldots ,\rho_n \in R  \mbox{ ( }  \rho_n \circ  
\ldots \circ \rho_1 (x) = y ) \} , 
\] 
where in the case when there exists no $n$ as in the formula above we put $d_R (x, y) = \infty$. 
Then viewing $d_R$ as a metric, for any finite $F\subset X$ we have: 
\[ 
|N_1 (F)| = |R\circ F|\geq (1+\frac{1}{n})|F|.
\] 
Choose $n_1\in \mathbb{N}$ such that $(1+\frac{1}{n})^{n_1}\geq 3$ and set $K=R^{n_1}$. 
So for any finite $F\subset X$ we have $|N_{n_1}(F)| = |K\circ F|\geq 3|F|$.

Now note that the the set of edges of the bipartite graph $\Gamma (X)=(\mathbb{N}, \mathbb{N}, E)$, defined in the proof of Theorem \ref{pdx} consists of all pairs $(x,y) \in \mathbb{N} \times \mathbb{N}$ 
with $y \in K\circ x$, where $x,y$ are viewed as elements of $X$ under the identification $X = \mathbb{N}$. 
Since $G$ consists of computable permutations and $K$ is finite, 
the graph $\Gamma(X)$ is computably bipartite. 
Since the degree of every vertex is computable 
(by application of $K$), the graph is highly computable. 
Exactly as in the proof of Theorem \ref{pdx} we see that the graph $\Gamma(X)$ satisfies $c.e.H.h.c.(2)$ with respect to $h= \mathsf{id}$. 
By virtue of the Effective Hall Harem Theorem, we deduce the existence of a computable perfect $(1,2)$-matching $M$ in $\Gamma_R(X)$. 
In other words, there is a computable surjective 2-to-1 map 
$\phi : \mathbb{N} \rightarrow \mathbb{N}$ such that for any $n\in \mathbb{N}$ 
there is  $g\in K$ with $n = g \circ \phi(n)$.   

Repeating the proof of Theorem \ref{pdx} define functions $\psi_1, \psi_2$ as follows: 
\[ 
\left\{
\begin{array}{r}
\psi_1(n)=\min(n_1,n_2) \\
\psi_2(n)=\max(n_1,n_2)
\end{array}\right. , \text{ where } \phi(n_1)=n=\phi(n_2), n_1\neq n_2.
\] 
Since the function $\phi$ realizes a computable perfect $(1,2)$-matching, 
both $\psi_1$ and $\psi_2$ are computable. 
Moreover, they preserve $\langle K \rangle$-orbits.  

Define $\theta_1(n)$ to be $g\in K$ with $\psi_1(n) =g \circ n$, 
and $\theta_2(n)$ to be $h\in K$ with $\psi_2(n)= h\circ n$. 
Observe that $\theta_1$, $\theta_2$ can be chosen computable and 
$\theta_1(n), \theta_2(n)\in K$ for all $n\in \mathbb{N}$. 

For each $k\in K$ define sets $A_k$ and $B_k$ in the following way:
\[ 
A_k=\{n\in\mathbb{N}: \theta_1(n)=k\},\; 
B_k=\{n\in\mathbb{N}: \theta_2(n)=k\}.
\] 
It is clear that these sets are computable and
\[ 
X=\bigsqcup\limits_{k\in K}A_k=\bigsqcup\limits_{k\in K}B_k.
\] 
For each $n\in A_k$, the value $\psi_1(n)$ is $k\circ n$. 
Thus $\psi_1(\mathbb{N})=\bigsqcup\limits_{k\in K}k\circ A_k$. 
Similarly we can show that 
$\psi_2(\mathbb{N})=\bigsqcup\limits_{k\in K}k\circ B_k$. 
Since $\mathbb{N}=\psi_1(\mathbb{N})\bigsqcup\psi_2(\mathbb{N})$, we have
\[ 
X = \Big(\bigsqcup\limits_{k\in K}k\circ A_k\Big)\bigsqcup\Big(\bigsqcup\limits_{k\in K}k\circ B_k\Big).
\] 
Therefore $(K,(A_k)_{k\in K}, (B_k)_{k\in K})$ is an effective paradoxical decomposition of the action of $G$ on $X$.
\end{proof}

\begin{rem} 
Groups of computable permutations of $\mathbb{N}$ are becoming an attractive object of investigations in computable algebra. 
We recommend the survey article \cite{MS} and the recent paper  
of the second author \cite{I}. 
Theorem \ref{EPDthm} shows how naturally these groups appear in computable amenability. 
\end{rem}

\section{Complexity of paradoxical decompositions} 

The approach of this section is similar to that in \cite{khmi}.  
Throughout the section, we assume that $G$ is a computable group. 
We then identify $G$ with $\mathbb{N}$ and regard  multiplication of $G$ and the inverse as computable functions $\mathbb{N}^2 \rightarrow \mathbb{N}$ and 
$\mathbb{N} \rightarrow \mathbb{N}$ respectively. 
Such a realization of $G$ is called a {\em computable presentation} of $G$. 
For simplicity we assume that $1$ is the neutral element of $G$.  
The expression  $x^{-1}$ means the inverse in $G$.   

Note that for any $g\in G$ the function $g \cdot x$, $x\in G$, defines a computable permutation on $\mathbb{N}$. 
In particular the left action of $G$ on $G$ is by computable permutations of $\mathbb{N}$. 

\begin{df} \label{EPD2} 
The computable group $G$ has a computable paradoxical decomposition, if the left action of $G$ on $G$ has a computable paradoxical decomposition. 
\end{df}

By Theorem \ref{EPDthm} (and its proof) we have the following statement. 

\begin{cor} \label{EPDgrp}
Let $K_0$ be a finite subset of $G$ and suppose there is $n\in \mathbb{N}\setminus \{ 0\}$ such that the following condition is satisfied:    
\begin{itemize} 
\item for any finite subset $F\subset G$, there exists $k \in K_0$ with $\frac{|F\setminus kF|}{|F|}\geq \frac{1}{n}$.  
\end{itemize} 
Let $n_1$ be such that $(1+\frac{1}{n})^{n_1} \ge 3$.  
Then the subset $K = (K \cup K^{-1} )^{n_1}$ defines a computable paradoxical decomposition as in Definition \ref{EPD}.
\end{cor}

In particular if $G$ is a computable non-amenable group then it has a 
computable paradoxical decomposition. 
This corollary leads to the following definition. 

\begin{df}\label{elf}
Let 
\[ 
\mathfrak{W}_{BT}=\left\{K\subset G \mbox{ is finite : } \exists n\in \mathbb{N} \; (\forall \mbox{ finite } F \subset G)(\exists k \in K)\left(\frac{|F\setminus kF|}{|F|}\geq \frac{1}{n}\right)\right\}.
\] 
We call $\mathfrak{W}_{BT}$ the set of witnesses of the Banach-Tarski paradox.
\end{df}

\begin{pr}
For any computable group the family $\mathfrak{W}_{BT}$ 
belongs to the class $\Sigma^{0}_{2}$ of the Arithmetical Hierarchy.
\end{pr}

\begin{proof}
Since the group $G$ is computable, for any finite subsets  $K$, $F$ of $G$, and any $n\in\mathbb{N}$, we can effectively check if the inequality 
$\frac{|F\setminus kF|}{|F|}< \frac{1}{n}$ holds for all $k\in K$. 
Therefore, the set of triples $(n,K,F)$ such that 
$ \frac{|F\setminus kF|}{|F|} < \frac{1}{n}$ holds for all $k\in K$ is computably enumerable, i.e. belongs to $\Sigma^0_1$. 

Since the projection of this set to the first two coordinates is also computably enumerable, the set 
\[ 
\mathfrak{W}_{BT}'=\{(K,n): (\forall \mbox{ finite }F \subset \Gamma)(\exists k\in K)(\frac{|F\setminus kF|}{|F|}\geq \frac{1}{n})\}
\]  
belongs to the class $\Pi_{1}^{0}$. 
The set $\mathfrak{W}_{BT}$ consists of all finite subsets $K\subset G$ such that there exists  $n\in \mathbb{N}$ with $(K,n)\in \mathfrak{W}_{BT}'$.  
Thus $\mathfrak{W}_{BT}$ belongs to the class $\Sigma^0_2$.
\end{proof}
\bigskip

It is well-known that a finitely generated free group has a computable presentation.  
We consider the following theorem as the most natural example where  the set $\mathfrak{W}_{BT}$ is computable.

\begin{theorem}\label{fg}
The family $\mathfrak{W}_{BT}$ is computable for any finitely generated free group.
\end{theorem}

Before the proof of this theorem we give some reformulation of witnessing. 
This observation belongs to M. Cavaleri. 
It simplifies our original argument. 

\begin{pr}\label{fif} 
Let $G$ be a group and $K$ be a finite subset of $G$.  
Then $K \in \mathfrak{W}_{BT}$ if and only if $\langle K \rangle$ is a non-amenable subgroup of $G$. 
\end{pr}

\begin{proof} 
The necessity holds by F\o lner's definition of amenability. 
Assume that $K \notin \mathfrak{W}_{BT}$. 
It follows that for every $n$ there exists a set $F_n$ such that 
\[ 
(\forall k \in K)\left(\frac{|F_n \setminus kF_n|}{|F_n|}\leq \frac{1}{n}\right) . 
\] 
In order to show that $\langle K \rangle$ is amenable 
we follow the proof of Proposition 9.2.13 from \cite{cor}. 
Take any $n\in \mathbb{N}$. 
Put $m=n|K|$. 
Let us show that there exists $t_0\in G$ such that the set $F_mt_0^{-1} \cap \langle K \rangle =\{ k\in \langle K \rangle: kt_0 \in F_m  \}$ is  $\frac{1}{n}$-F\o lner for $K$. 
Let $T\subset G$ be a complete set of representatives of the right cosets of $\langle K \rangle$ in $G$. 
Clearly, every $g\in G$ can be uniquely written in the form $g=ht$ with $h\in \langle K \rangle$ and $t\in T$. 
We then have:

\begin{equation}\label{el1}
|F_m|= \sum\limits_{t\in T}|
F_mt^{-1} \cap \langle K \rangle| . 
\end{equation}
For every $x\in K$, we have $xF_m= \bigsqcup\limits_{t\in T}(xF_mt^{-1} \cap \langle K \rangle)t$, hence:

\[ 
xF_m\setminus F_m = \bigsqcup\limits_{t\in T}((xF_mt^{-1} \cap \langle K \rangle)\setminus (F_mt^{-1} \cap \langle K \rangle))t.
\] 
This gives us:

\begin{equation}\label{el2}
|xF_m\setminus F_m|= \sum\limits_{t\in T}|(xF_mt^{-1} \cap \langle K \rangle)\setminus (F_mt^{-1} \cap \langle K \rangle)|.
\end{equation}
Since for all $x\in K$,
\[ 
|xF_m\setminus F_m|\leq \frac{|F_m|}{m},
\] 
using (\ref{el1}) and (\ref{el2}), we get

\begin{align*}
&\sum\limits_{t\in T}|(KF_mt^{-1} \cap \langle K \rangle)\setminus (F_mt^{-1} \cap \langle K \rangle)|=\\&\sum\limits_{t\in T}|\bigcup\limits_{x\in K} ((xF_mt^{-1} \cap \langle K \rangle)\setminus (F_mt^{-1} \cap \langle K \rangle))|\leq \frac{|K|}{m}\sum\limits_{t\in T}|
F_mt^{-1} \cap \langle K \rangle| . 
\end{align*}
By the pigeonhole principle, there exists $t_0\in T$ such that 
\[ 
|(KF_mt_0^{-1} \cap \langle K \rangle)\setminus (F_mt_0^{-1} \cap \langle K \rangle)|\leq \frac{1}{n}|F_mt_0^{-1} \cap \langle K \rangle| . 
\] 
Clearly $F_mt_0^{-1} \cap \langle K \rangle$ is an $\frac{1}{n}$-F\o lner set with respect to $K$. 
Since $n$ was arbitrary, $\langle K \rangle$ is amenable. 
This finishes the proof. 
\end{proof}

\bigskip

\begin{proof} {\em (Theorem \ref{fg}).} 
Let $\mathbb{F}$ be a finitely generated free group under the standard presentation. 
Since it is computable, the equation $xy = yx$  can be effectively verified for every $x,y\in \mathbb{F}$. 
We will show that $K\in \mathfrak{W}_{BT}$ if and only 
if there exist $ x,y\in K$ such that $ xy\neq yx$. 
This will give the result. 

$(\Rightarrow)$
Let us assume that $ xy=yx$ for every $x,y\in K$. 
Since $\mathbb{F}$ is a free group, there exists $z\in\mathbb{F}$ such 
that all words from $K$ are powers of $z$. 
Since the subgroup $\langle z \rangle$ is cyclic, the subgroup $\langle K\rangle$ is amenable and 
for every $n$ there is a finite set $F$, which is an $\frac{1}{n}$-F\o lner with respect to $K$. 
Clearly $K\notin \mathfrak{W}_{BT}$. 

$(\Leftarrow)$
Let us assume that there exist $x,y\in K$ with $ xy\neq yx$. 
Then $x,y$ generate a free subgroup of $\mathbb{F}$ of rank $2$. 
By Proposition \ref{fif} there is a natural number $n$ 
such that $\mathbb{F}$ does not contain $\frac{1}{n}$-F\o lner subsets with respect to both $\{ x,y\}$ and $K$. 
\end{proof}

We add few words concerning the following question. 
\begin{itemize} 
\item Are there natural examples with 
non-computable  $\mathfrak{W}_{BT}$? 
\end{itemize}

In \cite{DuI} (see also \cite{Du}) we give an example of a finitely presented group, say $H_{nA}$, with decidable word problem such that detection of all finite subsets of $H_{nA}$ which generate amenable subgroups, is not decidable. 
Applying Proposition \ref{fif} we see that the set $\mathfrak{W}_{BT}$ 
is not computable in this group.  
In \cite{DuI} we used slightly involved methods of computability theory. 
It can be also derived from \cite{DuI} and \cite{Du} that 
when a computable group $G$ is {\em fully residually free} \cite{kap}, 
the corresponding set $\mathfrak{W}_{BT}$ is computable.

\section*{Acknowledgements} 
\begin{itemize} 
\item
The authors are grateful to M. Cavaleri, T. Ceccherini-Silberstein and L. Ko{\l}odziejczyk for reading the paper and helpful remarks. 
In particular, the idea of Proposition \ref{fif} belongs to M. Cavaleri. 
\end{itemize}

\end{document}